\documentclass{amsart}
\usepackage{amstext, amssymb, amsthm, amsfonts, latexsym,bbm,tikz}
\usetikzlibrary{matrix}
\usepackage{varioref}
\usepackage{mathtools}

\usepackage{enumerate}
\usepackage{todonotes}
\usepackage{bbm}
\usepackage[utf8]{inputenc} 
\usepackage[T1]{fontenc}
\usepackage[version=4]{mhchem}
\usepackage{ulem}

\usepackage{comment}
\usepackage{tikz-cd}
\usepackage{hyperref}

\newtheorem{theorem}{Theorem}[section]
\newtheorem{lemma}[theorem]{Lemma}
\newtheorem{corollary}[theorem]{Corollary}
\newtheorem{proposition}[theorem]{Proposition}

\theoremstyle{definition}
\newtheorem{definition}[theorem]{Definition}
\newtheorem{example}[theorem]{Example}

\theoremstyle{remark}

\newcommand{\N}{\mathbb{N}}
\newcommand{\Z}{\mathbb{Z}}

\newcommand{\cR}{\mathcal{R}}

\newcommand{\ty}{\widetilde{y}}
\newcommand{\cS}{\mathcal{S}}

\newcommand{\cU}{\mathcal{U}}
\newcommand{\cF}{\mathcal{F}}

\newcommand{\supp}{\mathrm{supp}}
\newcommand{\cl}{\mathrm{cl}}

\numberwithin{equation}{section}

\newcommand\black[1]{\textcolor{black}{#1}}

\begin{document}

\title{On the sum of chemical reactions}
\hfill\break
\author[Linard Hoessly]{Linard Hoessly}

\address{Department of Mathematical Sciences, University of Copenhagen, Denmark}
\email{linard.hoessly@hotmail.com}
\author[Carsten Wiuf]{Carsten Wiuf}

\address{Department of Mathematical Sciences, University of Copenhagen, Denmark}
\email{wiuf@math.ku.dk}

\author[Panqiu Xia]{Panqiu Xia}

\address{Department of Mathematical Sciences, University of Copenhagen, Denmark}
\email{px@math.ku.dk}

\date{}
\subjclass[2010]{}

\begin{abstract}
It is standard in chemistry to represent a sequence of reactions by a single overall reaction, often called a complex reaction in contrast to an elementary reaction. Photosynthesis $6 \text{CO}_2+6 \text{H}_2\text{O} \ce{->} \ \text{C}_6\text{H}_{12}\text{O}_6$ $+\ 6 \text{O}_2$ is an example of such complex reaction. We introduce a mathematical operation that corresponds to summing two chemical reactions. Specifically, we define an associative and non-communicative operation on the product space $\N_0^n\times \N_0^n$ (representing the reactant and the product of a chemical reaction, respectively). The operation models the overall effect of two reactions happening in succession, one after the other. We study the algebraic properties of the operation and apply the results to stochastic reaction networks, in particular to reachability of states, and to reduction of reaction networks.
\end{abstract}

\maketitle


\section{Introduction}

Systems of chemical reactions are commonly modeled by reaction networks (RNs) \cite{gardiner,Feinberg}. RNs provide a comprehensive mathematical framework for modelling systems of interacting species that is not only used in chemistry and biophysics, but also in mathematical genetics \cite{ewens2004}, epidemiology \cite{epid_CRN}, cellular and systems biology \cite{wilkinson}, and sociology \cite{sociology}. Notable examples include the  Lotka-Volterra predator-prey system \cite{murray}, and the SIR model \cite{AndersonMay}.

If a series of reactions occur one by one, it is natural to ask for the  overall effect of the reactions, that is, the {\it sum} (in some sense) of the reactions. In fact, it is standard in chemistry  to summarize reactions into a single overall or {\it complex}  reaction, in contrast to elementary  reactions. As an example, photosynthesis consists of a  sequence of reactions, summarized into the complex reaction
$$6\ \text{CO}_2+6\ \text{H}_2\text{O}\ \ce{->} \ \text{C}_6\text{H}_{12}\text{O}_6+6\ \text{O}_2$$
\cite{photobiol}.
Graphical treatment of such sequences of reactions, that is of complex  reactions, have a long history in the chemical literature, see e.g. \cite{christiansen,temkin,temkinBook,sakamoto}. Here, we provide the  mathematical framework for {\it adding} such sequences of reactions. 

As an example, consider an RN  describing single gene expression \cite{Thattai8614},
$$ 0 \ce{<=>} R,\quad P\ce{->} 0,\quad R\ce{->} R+P,$$
where $R$ denotes an mRNA molecule and $P$ a protein. The mRNA is freely produced from a gene (the reaction $0\ce{->}R$), the protein is translated from the mRNA, and both protein and mRNA are degraded. Modelled as a {\it discrete  system}, the state space is  $\N_0^2$, pairs of integers repesenting the number of $R$ and $P$ molecules, respectively. Jumps between states are given by reaction vectors, for example, a {\it direct} jump from $(k,\ell)\in\N_0^2 $ to $(k,\ell+1)$ is possible by means of the reaction $R\ce{->} R+P$, if the number $k$ of $R$ molecules is $\ge 1$.
If $k=0$, then the sequence of reactions $0 \ce{->} R$, $R\ce{->} R+P$,  $R\ce{->}0$
 will take the system from the state $(0,\ell)$ to the state $(0,\ell+1)$. In that case, one might describe the overall effect (the {\it sum}) of the reaction sequence as $0\ce{->}P$. As the $R$ molecule is created in the first reaction and degraded in the third, it cancels in the sum. Using similar arguments, one can conclude that the set of reachable states from any state $(k,\ell)\in\N_0^2$ is all of $\N_0^2$.
 
When  the number of molecules of each species (here $R,P$) is low (as is often the case if the system is embedded into a cellular environment), it is appropriate to consider the system as a  discrete stochastic system
in $\N_0^n$. If so, it is standard to model the changes in molecule counts by a continuous-time Markov chain \cite{approx_kurtz,anderson1}. For example, with stochastic mass-action kinetics, the  propensities for the reactions to take place have the form
$$\lambda_{y\to y'}(x)=\kappa_{y\to y'}\frac{x!}{(x-y)!}\mathbbm{1}_{\{z\colon z\geq y\}}(x),$$
where $\kappa_{y\to y'}$ is a positive rate constant and $z!:=\prod_{i=1}^nz_i!$ for $z\in \N^n_0$.  A first step in the analysis of a stochastic dynamical RN is to understand the structure of the reachable sets and the irreducible classes; that is, to understand whether the system is confined to subspaces of $\N_0^n$, is absorbed in certain states, etc, depending on the initial state of the system.

In the following, we examine a binary sum operation on $\N_0^n\times\N_0^n$, that describes the addition of two chemical reactions, as illustrated in the single gene expression RN above. We study the operation's algebraic properties and its applications. In terms of applications, we exhibit connections to discrete RNs and reachability properties, and  to  reductions of discrete RNs. Common to these applications is the idea of reactions happening in succession, one after the other. 

Reactions  often occur at different time-scales \cite{KK13}. This has led to various methods for {\it reduction} of  RNs, where {\it fast} reactions and/or species are eliminated (in a precise mathematical sense). These methods are generally not qualitative (or graphical) {\it per se}, but quantitative,  and depend on whether the dynamics of the RN is stochastic \cite{KK13,CW16,gardiner} or deterministic \cite{BOWEN_1,feliu2019quasisteady}. 
 If the reactions in a sequence occur at a fast  rate (that is, with high intensity), it is natural to assume no other reactions take place before the last reaction of the sequence has occured.  Rather than describing the entire sequence of reactions, one might summarize the sequence by a single complex reaction, the overall effect. In a sense, this complex reaction is obtained by {\it contraction}. We define contraction of reactions through the defined sum operation and subsequently define reduced RNs. These constructs are essentially graphical in nature. We show that they relate to stochastic approaches for reduction of RNs, in particular to reduction by  elimination of so-called {\it intermediate} and {\it non-interacting} species \cite{CW16,HW2021}. 

Furthermore, we study graphical properties of  the state space of discrete RNs concerning the operation we introduce. We show that reachability can be expressed via the sum operation, and in particular that the closure of the sum operation determines reachability.  

In Section \ref{sec.agb}, we define the sum of two reactions and study the properties of the operation. In Section \ref{sec.appl}, we specialize to reaction networks and reachability properties. Finally, in Sections \ref{sec:reduction} and \ref{sec:red_reach}, we study reductions of RNs. In the latter section, we draw on Section \ref{sec.agb} and study conditions that ensure that the reduction leads to reversible (or weakly reversible, or essential) RNs. 
\subsection*{Acknowledgements}
The work presented in this article is supported by Novo Nordisk Foundation,
grant NNF19OC0058354. LH acknowledges funding from the Swiss National Science Foundations Early Postdoc.Mobility
grant (P2FRP2\_188023).

\section{Algebra on \texorpdfstring{$\N_0^n\times \N_0^n$}{}}\label{sec.agb}

Denote by $\Z$ the set of integers, and by $\N_0$ the set of  non-negative integers. Let $n$ be a positive integer. For  $x=(x^1,\dots,x^n),$ and $y=(y^1,\dots, y^n)$
in $\Z^n$, we write $x\leq y$ if $x^i\leq y^i$ for  $i=1,\dots,n$, and $x<y$ if $x\leq y$ and $x\neq y$. We also use the notation $x\ll y$ if $x^i<y^i$ for $i=1,\ldots,n$.  Furthermore, we let $x\vee y=(x^1\vee y^1,\dots, x^n\vee y^n)=(\max\{x^1,y^1\},\dots,\max\{x^n,y^n\})$ be the componentwise maximum, and let $x\wedge y=(x^1\wedge y^1,\dots, x^n\wedge y^n)=(\min\{x^1,y^1\},\dots,\min\{x^n,y^n\})$ be the componentwise minimum.

\begin{definition}\label{add}
Let $r_1=(y_1,y_1'), r_2=(y_2,y_2')\in \N_0^n\times \N_0^n$. Then $r_1\oplus r_2=(y,y')$ is the  element in $\N_0^n\times \N_0^n$ given by 
$y=y_1+0\vee(y_2-y_1')$
and $y'=y_2'+0 \vee(y_1'-y_2).$
\end{definition}

\begin{proposition}\label{pro}
$(\N_0^n\times \N_0^n,\oplus)$ forms a non-commutative monoid with identity $(0,0)$.
\end{proposition}
\begin{proof}
It is straightforward to see that $\oplus$ is a stable operation on $\N_0^n\times \N_0^n$ with $(0,0)\oplus r=r\oplus (0,0)=r$ for all $r\in \N_0^n\times \N_0^n$. Let $y_1\neq y_2\in \N_0^n$. Then
\[
(y_1,y_2)\oplus (y_2,y_1)=(y_1,y_1)\neq (y_2,y_2)=(y_2,y_1)\oplus (y_1,y_2),
\]
hence $\oplus$ is non-commutative.
To prove associativity, we assume without loss of generality, that  $n=1$. For $n>1$, it follows by looking at each coordinate independently. Let $r_i=(y_i,y_i')$ for  $i=1,2,3$. Furthermore, let $(y,y')=(r_1\oplus r_2)\oplus r_3$ and $(\ty,\ty')=r_1\oplus (r_2\oplus r_3)$. Then,
\begin{align*}
y&=y_1+0\vee (y_2-y_1')+0\vee (y_3-y_2'-0\vee (y_1'-y_2)),\\
y'&=y_3'+0\vee (y_2'+0\vee (y_1'-y_2)-y_3),\\
\ty&=y_1+0\vee (y_2+0\vee(y_3-y_2')-y_1'),\\
\ty'&=y_3'+0\vee (y_2'-y_3)+0\vee(y_1'-y_2-0\vee (y_3-y_2')).
\end{align*}
By distinguishing the following three cases a) $y_2\geq y_1'$, b) $y_2< y_1'$ and $y_3\geq y_2'$,  and c) $y_2< y_1'$ and $y_3<y_2'$, it is easy to verify that $y=\ty$. A similar argument gives $y'=\ty'$. The proof  is complete.
\end{proof}

The sum operation reduces to standard  addition in $\N_0^n$ on the two axis, and it is the component-wise maximum (addition in max-plus algebras) on the diagonal. The proof of the next result is straightforward and omitted.
\begin{proposition}\label{properties_sum}
Let $r_1=(y_1,y_1'), r_2=(y_2,y_2')\in \N_0^n\times \N_0^n$. Then,

\begin{enumerate}[(i)]
\setlength\itemsep{1em}
    \item  If $y_1=y_2=0$,   then $ r_1\oplus r_2=(0,y_1'+y_2')$. 
    \item  If $y_1'=y_2'=0$,    then $ r_1\oplus r_2=(y_1+y_2,0)$.
    \item  If $y_1=y_1'$, $y_2=y_2'$, then $r_1\oplus r_2=(y_1\vee y_2,y_1'\vee y_2')$.
     \item \label{properties_sum4} $(y_1,y_2')\leq r_1\oplus r_2\leq (y_1+y_2,  y_1'+y_2')$.   Furthermore, the first equality holds if and only if $y_1'=y_2$ and the second equality holds if and only if $y_1'\wedge y_2=0$. 
\end{enumerate}
\end{proposition}

The next statement characterizes the sum operation. 

\begin{proposition}\label{unique_char}
Let $r_1=(y_1,y_1'), r_2=(y_2,y_2')\in \N_0^n\times \N_0^n$ and  $r_1\oplus r_2=(y,y')$. Then
\begin{enumerate}[(i)]
\setlength\itemsep{1em}
\item\label{unique_char1} $y'-y= (y_1'-y_1)+(y_2'-y_2)$,
\item\label{unique_char2} For $x\in \N_0^n\colon x\geq y$ if and only if $x\ge y_1$ and $x+(y_1'-y_1)\geq y_2$,
    \item\label{unique_char3} For $x\in \N_0^n\colon x\geq y'$ if and only if $x\geq y_2'$ and $x+(y_2-y_2')\geq y_1'$.
\end{enumerate}
Oppositively, if \eqref{unique_char1} and \eqref{unique_char2}, or alternatively, if \eqref{unique_char1} and \eqref{unique_char3}, are fulfilled for some  operation $\oplus$ on $\N_0^n\times \N_0^n$, then it is the sum operation in Definition \ref{add}.
\end{proposition}

\begin{proof}
\eqref{unique_char1} The claim  follows from $y'-y=y_2'+0\vee (y_1'-y_2)-y_1-0\vee (y_2-y_1'),$
as $0\vee (y_1'-y_2)-0\vee (y_2-y_1')=y_1'-y_2.$
\eqref{unique_char2} It is a direct consequence of $y=y_1+0\vee (y_2-y_1')=y_1\vee (y_1+y_2-y_1')$. \eqref{unique_char3} follows similarly.

Oppositely, assume \eqref{unique_char1} and \eqref{unique_char2} are fulfilled for some operation $\oplus$. Then, for any $(y_1,y_1')\oplus (y_2,y_2')=(y,y')$, it holds that $x\geq y$ if and only if $x\geq y_1$ and $x+y_1'-y_1\geq y_2$, that is, $x\geq y_2-y_1'+y_1$. This implies that $y=y_1+0\vee (y_2-y_1')$. Combining this fact with \eqref{unique_char1}, we get $y'=y_2'+0\vee (y_1'-y_2)$. If \eqref{unique_char1} and \eqref{unique_char3} are fulfilled for some operation $\oplus$, the proof is similar. It completes the proof.
\end{proof}

We next introduce an equivalence relation on $\N_0^n\times \N_0^n$ under which the corresponding quotient set is a commutative group.
\begin{definition}
Let $r_1=(y_1,y_1'),r_2=(y_2,y_2')\in \N_0^n\times \N_0^n$. Then,  $r_1$ and $r_2$ are equivalent, denoted by $r_1\sim r_2$, if $y_1'-y_1=y_2'-y_2$.
\end{definition}

The following theorem follows by definition and Proposition \ref{unique_char}\eqref{unique_char1}. 

\begin{theorem}\label{cmgp}
$((\N_0^n\times \N_0^n)/\sim, \oplus)$ forms a commutative group.
\end{theorem}

The next proposition shows that {\it subtraction} might be defined on $\N_0^n\times \N_0^n$ instead of the quotient space, in some situations.

\begin{proposition}\label{prop_elim}
Let $r_1=(y_1,y_1'),\widetilde{r}_1=(\ty_1,\ty_1'), r_2=(y_2,y_2')\in \N_0^n\times \N_0^n$. The following properties hold.
\begin{enumerate}[(i)]
\setlength\itemsep{1em}
\item \label{subt1} Suppose that $r_1\oplus r_2=\widetilde{r}_1\oplus r_2$ and $y_2\ll  y_1'$. Then, $r_1=\widetilde{r}_1$.

\item \label{subt2}Suppose that $r_2\oplus r_1=r_2\oplus \widetilde{r}_1$ and $y_1\ll  y_2'$. Then, $r_1=\widetilde{r}_1$.
\end{enumerate}
\end{proposition}

\begin{proof}
We show property \eqref{subt1}.  The proof of property \eqref{subt2} is similar. If $y_2\ll y_1'$, then by definition and the assumption that  $r_1\oplus r_2=\widetilde{r}_1\oplus r_2$, we get
\[y_1=y_1+0\vee (y_2-y_1')=\ty_1+0\vee (y_2-\ty_1')\]
and
\[ y_2'+y_1'-y_2=y_2'+0\vee (y_1'-y_2)=y_2'+0\vee (\ty_1'-y_2).\]
From the second equation,  we have $0\ll y_1'-y_2=0\vee (\ty_1'-y_2)$. Hence, $y_1'-y_2=\ty_1'-y_2$, which implies $\ty_1'= y_1'$. As a result of the first equality, we have $y_1=\ty_1$. The proof  is complete.
\end{proof}

The condition $y_2\ll  y_1'$  in Proposition \ref{prop_elim}\eqref{subt1} (as well as that in (ii)) cannot be weakened, which can be seen by example.

Proposition \ref{pro} allows us to define the (non-commutative) summation of a finite sequence of elements in $\N_0^n\times \N_0^n$,
\[
\oplus_{i=1}^m r_i=r_1\oplus r_2\oplus \dots \oplus r_m.
\]
For any $r=(y,y')\in \N_0^n\times \N_0^n$, let $r^{-1}=(y',y)$ be the {\it inverse} of $r$. Then, 
$r\oplus r^{-1}=(y,y),$ and $r^{-1}\oplus r=(y',y')$.
The inverse is unique, and furthermore
\begin{align*}
\left(\oplus_{i=1}^m r_i\right)^{-1}=\oplus_{i=1}^{m}r_{m+1-i}^{-1}
\end{align*}
for $r_1,\ldots,r_m\in\N_0^n\times \N_0^n$, $m=1,2,\dots$.

\begin{corollary}\label{cor:m}
If $r_1=(y_1,y_1'),\dots, r_m=(y_m,y_m')\in \N_0^n\times \N_0^n$ and  $\oplus_{i=1}^mr_i=(y,y')$, then
\begin{itemize}
\setlength\itemsep{1em}
\item[(i)] $y'-y=\sum_{i=1}^my_i'-y_i$,
\item[(ii)] For $x\in \N_0^n\colon x\geq y$ if and only if  $x+\sum_{i=1}^k(y_i'-y_i)\geq y_{k+1}$ for  $k=0,1,\ldots,m-1$.
\end{itemize}
\end{corollary}

\begin{proof}
Let $r_{(m)}=\oplus_{k=1}^m r_k$. The corollary is then a consequence of Proposition \ref{unique_char} and  induction in $m$.
\end{proof}

A subset $A\subseteq\N_0^n\times \N_0^n$ is said to be {\it closed} (under $\oplus$) if for any $r_1,r_2\in A$, $r_1\oplus r_2\in A$ as well. 
Denote by $\cl(A)$ 
the {\it closure} of $A$, that is, the collection of all $r\in \N_0^n\times \N_0^n$ that can be represented as a {\it finite} sum of elements in $A$, including the empty sum  by convention, that is, $(0,0)\in \cl(A)$. Thus, $\cl(A)$ is the smallest closed set containing $A\cup\{(0,0)\}$, namely,  $\cl(A)$ is a subset of any closed set $A'\cup\{(0,0)\}$ with $A\subseteq A'$. 

We next introduce several notions related to  {\it reversibility}. The concepts to be introduced are analogous to concepts in  reaction network theory, cf. \cite{Feinberg,Cappelletti}. In particular, the term {\it essential} comes from Markov chain theory, but it is also used in  
reaction network theory \cite{Cappelletti}. It is also equivalent to  {\it recurrent}, defined in \cite{jmb-14-pauleve-craciun-koeppl} (see below), which is different from recurrent in Markov chain theory.

\begin{definition}\label{def_CRNs}
Let $A$ be a subset of $ \N_0^n\times \N_0^n$. We say  
\begin{enumerate}[(i)]
\setlength\itemsep{1em}
    \item \black{$r\in A$  is {\it reversible} in $A$ if $r^{-1}\in A$.} The set  $A$ is {\it reversible}, if  $r\in A$ implies $r^{-1}\in A$.

    \item \label{def_crns2} \black{ $r\in A$ is {\it weakly reversible} in $A$, if there exist a sequence of elements $r_1=(y_1, y_1'),\dots,r_m=(y_m, y_m')\in A$, such that $y_{k-1}'=y_k$ for $k=2,\dots,m$, and $\oplus_{i=1}^mr_i=r^{-1}$. The set  $A$ is  {\it weakly reversible}, if for any $r\in A$, $r$ is weakly-reversible in $A$.}

    \item \label{def_crns3} $A$ is  {\it essential}, if $\cl(A)$    is reversible.
\end{enumerate}  
\end{definition}
By Proposition \ref{properties_sum}\eqref{properties_sum4}, we have $\oplus_{i=1}^mr_i=(y_1,y_m')$ in Definition \ref{def_CRNs}\eqref{def_crns2}. Clearly the following implications hold by definition.

\begin{lemma}
Let $A$ be a subset of $ \N_0^n\times \N_0^n$. Then,
$$A \text{ is reversible}\implies A  \text{ is  weakly reversible}\implies A  \text{ is  essential.}$$
\end{lemma}

\section{RNs and reachability}\label{sec.appl}

In this section, we combine the algebra defined in Section \ref{sec.agb} with reaction network theory  and present some reachability results. By definition, an RN is a subset $\cR\subseteq \N_0^n\times \N_0^n$,
containing no elements $r$ equivalent to $(0,0)$. For convenience, we allow $\cR$ to be infinite, though this is not standard in the literature  \cite{Feinberg}. We use standard terminology for reaction networks, and refer to an element $r=(y,y')\in\cR$ as a  {\it reaction},  $y$ as the {\it reactant}  and $y'$  as the {\it product} of this reaction. The species of $y$ are {\it degraded} and those of $y'$ are {\it produced}.
 Furthermore, as is standard in the literature, we consider an RN as a graph, writing $y\ce{->} y'$ for $(y,y')\in\cR$, 
 and $y\ce{<=>} y'$ for $(y,y'),(y',y)\in\cR$.
 
  For  $i=1,\dots, n$, we denote by $S_i$ the $i$th unit vector in $\N_0^n$, such that $\cS=\{S_1,\dots, S_n\}$ forms a complete basis of $\N_0^n$. For  $y=(y^1,\dots, y^n)\in \N_0^n$, we thus have $y=\sum_{i=1}^n y^i S_i$.   We refer to $S_i$ as the $i$th {\it species}, and the component $y^i$  as the {\it stoichiometric coefficient} of the species $S_i$ in $y$.

\begin{example}\label{ex_2MM}
Consider a two-substrate mechanism \cite{bowden}, 
$$E+A\ce{<=>} EA,\quad EA+P\ce{->}EQ\ce{->}E+Q,$$
where $E$ is an enzyme catalysing the conversion of a substrate $A$ to another substrate $Q$ through a third intermediate substrate $P$. The molecules $EA$ and $EQ$ are referred to as transient (or intermediate) complexes.

Using the notation introduced above, let $S_1=(1,0,0,0,0,0)=E$, $S_2=(0,1,0,0,0,0)=A$, $S_3=(0,0,1,0,0,0)=EA$, $S_4=(0,0,0,1,0,0)=P$, $S_5=(0,0,0,0,1,0)=EQ$ and $S_6=(0,0,0,0,0,1)=Q$. Then, we might write the reactions as follows,
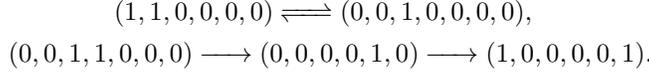
\begin{align*}
&\qquad \qquad (1,1,0,0,0,0)\ce{<=>} (0,0,1,0,0,0,0),\\
& (0,0,1,1,0,0,0)\ce{->} (0,0,0,0,1,0)\ce{->} (1,0,0,0,0,1).
\end{align*}
For example, the species $E$ has stoichiometric coefficient $1$ in $(1,1,0,0,0,0)=S+E$.
\end{example}

In the stochastic theory of  RNs with {\it finite} number of reactions, the molecule counts follow a continuous-time Markov process $\{X(t)\}_{t\ge 0}$ with state space $\N_0^n$. Jumps  occur according to the ``firing'' of reactions: The reaction $y\ce{->} y'\in\cR$ has transition intensity $\lambda_{y \to y'}(x)$ and when it occurs the process jumps from state $x$ to state $x+y'-y$, where $y'-y$ is the {\it net gain} of the reaction \cite{anderson1}. 
The Markov process satisfies the following equation:
\begin{equation}\label{eq:markov}
P(X(t+\Delta t)=x+\xi|X(t)=x)=\sum_{y\to y'\in \cR\colon  y'-y  =\xi}\lambda_{y\to y'}(x)\Delta t+ o(\Delta t),
\end{equation}
for $\xi\in\Z^n$ and some initial count $X(0)=x_0\in\N_0^n$. As $\cR$ is finite, then \eqref{eq:markov} defines the process $\{X(t)\}_{t\ge 0}$ (provided the chain does not explode).

Generally, the transition intensities $\lambda_{y \to y'}\colon\N_0^n\to [0,\infty)$, for $y\ce{->} y'\in\cR$, are assumed to satisfy the {\it compatibility} condition
\begin{equation}\label{eq:x}
\lambda_{y\to y'}(x)>0\quad \iff \quad x\geq y,
\end{equation}
or the weaker condition
\begin{equation}\label{eq:x2}
\lambda_{y\to y'}(x)>0\quad \implies  \quad x\geq y.
\end{equation}
These have  natural interpretations: A reaction $y\ce{->} y'$ can occur (if and) only if the molecule counts are larger than or equal to $y$. Below, we adhere to \eqref{eq:x} and note that similar  statements (one-way implications) to those we derive can be achieved assuming \eqref{eq:x2} only.

A reaction $y\ce{->} y'\in \cR$ is said to be {\it active} on a state $x\in\Z^n$ if $\lambda_{y \to y'}(x)>0$, and  an ordered sequence of reactions $y_1\ce{->} y_1',\ldots, y_m\ce{->} y_m'\in\cR$ is said to be {\it active} on $x$ if
\begin{equation}\label{eq:xx}
\lambda_{y_k\to y_k'}\!\left(x+\sum_{i=1}^{k-1}y_i'-y_i\right)>0, \quad  k=1,\ldots, m,
\end{equation}
that is, if the sequence of reactions can happen in succession, one after the other. After each step  the molecule count is updated. 
In particular, an ordered  sequence of reactions is active on $x$ if and only if there is a positive probability that the Markov chain performs this sequence of reactions in the given order.

Assume the compatibility condition \eqref{eq:x} holds. Then,  \eqref{eq:xx} is equivalent to  $x+\sum_{i=1}^{k-1}(y_i'-y_i)\ge y_k$ for $k=1,\ldots,m$.
According to Proposition \ref{unique_char} and Corollary \ref{cor:m}, this provides the following interpretation of the sum operation.

\begin{corollary}\label{coroactive}
An ordered sequence of reactions $y_1\ce{->} y_1',\ldots, y_m\ce{->} y_m'\in\cR$ is active  on a state $x\in \N_0^n$, if and only if $x\ge y$, where $(y, y') =\oplus_{i=1}^m (y_i\ce{->} y_i')$.
\end{corollary}

For  a stochastic RN, {\it reachability} to a state $x'\in \N_0^n$ or the set of reachable states 
 from an initial state $x\in\N_0^n$, is often a main interest  \cite{Mreach}.
A state $x$ {\it leads to} a state $x'$ via an RN $\cR$, or equivalently, $x'$ is {\it reachable} from $x$ if there is an active ordered sequence of  $m\geq 0$ reactions  $y_1\ce{->} y_1',\ldots,y_m\ce{->} y_m'\in\cR$ such that $x'=x+\sum_{i=1}^my_i'-y_i$.  As a  consequence of Proposition \ref{unique_char}  and Corollary \ref{cor:m}, we can reformulate reachability of elements in $\N_0^n$ via $\cR$  as follows.

\begin{lemma}\label{leads} 
Let $\cR$ be an RN. A state $x\in\N_0^n$ leads to $x'\in\N_0^n$ if and only if there is  $(y,y')\in\cl(\cR)$ with $x\geq y$ and $x'=x+y'-y$; equivalently $(x,x')\geq (y,y')$ and $(x,x')\sim (y,y')$.
\end{lemma}

 Denote by $\cR(x)=\{x'\in \N_0^n| x\ \mathrm{leads\ to}\ x'\}$ the set of reachable states of $x\in \N_0^n$ via $\cR$.
\begin{corollary} 
For two reaction networks $\cR_1,\cR_2$ on the same set of species we have 
$$\cl(\cR_1)=\cl(\cR_2)\quad\implies \quad\mathrm{\ for\ all\ } x\in \N_0^n, \,\, \cR_1(x)=\cR_2(x).$$
\end{corollary}
Hence, having the same $\cl(\cR)$ for two RNs is in general stronger than having the same reachability sets for all initial states (in the latter case, the RNs are said to be structurally identical \cite{arxiv-20-wiuf-xu}). 

Say a reaction $y\to y'\in\cR$ has a {\it catalytic species} if there is a species $S_i$ such that $y^i> 0,(y')^i>0$.
Then an RN with no catalytic species is an RN where no reaction has a catalytic species. For RNs without catalytic species, the previous corollary can be strengthened.

\begin{theorem} 
For two reaction networks $\cR_1,\cR_2$ on the same set of species and  without catalytic species, we have 
$$\cl(\cR_1)=\cl(\cR_2)\quad\iff\quad \mathrm{\ for\ all\ } x\in \N_0^n, \,\, \cR_1(x)=\cR_2(x).$$
\end{theorem}

\begin{proof}
We only need to prove the right to left implication. By symmetry it is sufficient to prove $\cl(\cR_1)\subseteq\cl(\cR_2)$. Consider the set $B_0=\{r\in\N_0^{n}\times \N_0^n|r\sim r_0\}$, $r_0\in \cR_1$. Any element of $B_0$ takes the form $r=r_0+(y,y)\in\N_0^n\times \N_0^n$ for some $y\in\Z^n$. As there are no catalytic species, then $y\ge 0$ and $r\ge r_0$.

Let $r_0=(y_0,y_0')$. By definition, $y_0$ leads to $y_0'$ in $\cR_1$. As $\cR_1(y_0)=\cR_2(y_0)$,  then also  $y_0'\in\cR_2(y_0)$. Hence, $y_0$ leads to $y_0'$ in $\cR_2$, and by Lemma \ref{leads} there is an element $\widetilde{r}\in \cl(\cR_2)$ that realises this. By definition, $\widetilde{r}\sim r_0$ and $\widetilde{r}\in B_0$, hence $\widetilde{r}\geq r_0$ from above. By Lemma \ref{leads}, $r_0\geq \widetilde{r}$, hence $\widetilde{r}=r_0$ and $r_0\in \cl(\cR_2)$.
Now consider an arbitrary element $\widetilde{r}\in\cl(\cR_1)$,  given as $\widetilde{r}=\widetilde{r}_1\oplus \ldots \oplus \widetilde{r}_k$ with $\widetilde{r}_i\in \cR_1$, $i=1,\ldots,k$. We have $\widetilde{r}_i\in\cl(\cR_2)$ for $i=1,\ldots,k$, hence also $\widetilde{r}\in \cl(\cR_2)$ by the  closure property.
\end{proof}

 Finally, we characterize the property of being essential. Moreover, we prove the equivalence between essential RNs defined  in Definition \ref{def_CRNs}\eqref{def_crns3} and recurrent RNs defined  in  \cite{jmb-14-pauleve-craciun-koeppl}.

 \begin{proposition}\label{essn}
An RN $\cR$ is essential if and only if for  $x,x'\in\N^n_0$, if $x$ leads to $x'$, then $x'$ leads to $x$. 
\end{proposition}

\begin{proof}
If $\cR$ is an essential RN, then 
as a consequence of Lemma \ref{leads},  $x$ leads to $x'$ whenever $x'$ leads to $x$. Oppositely, assume that $\cR$ is  such that  $x$ leads to $x'$ whenever $x'$ leads to $x$ for all $x,x'\in \N_0^n$. Then, for any $r_0\in \cl(\cR)$, let $r_*=(y_*,y_*')\leq r_0$ be a minimal element of $\{r\in\cl(\cR)|\ r\sim r_0\}$ (which exists by Zorn's lemma, but is not necessarily unique). Note that by Lemma \ref{leads} we have that $y_*$ leads to  $y_*'$, hence by assumption also that $y_*'$ leads to  $y_*$. Then by Lemma \ref{leads} there is  $\widetilde{r}\in \cl(\cR)$ with $\widetilde{r}\leq r_*^{-1}$ and $\widetilde{r}\sim r_*^{-1}$,  which is equivalent to $\widetilde{r}^{-1}\leq r_*$ and $\widetilde{r}^{-1}\sim r_*$. 
Similarly, we can find $\widehat{r}\in\cl(\cR)$ with $\widehat{r}\leq \widetilde{r}^{-1}$ and $\widehat{r}\sim \widetilde{r}^{-1}$. Thus, we have $\widehat{r}\leq r_*$ and $\widehat{r}\sim r_*$. As $r_*$ is chosen to be a minimal element of $\{r\in\cl(\cR)|\ r\sim r_0\}$, this implies  $\widehat{r}=r_*$ and thus $\widetilde{r}=r_*^{-1}$. Finally, as $\cl(\cR)$ is a closed set, it is enough to check the equality $r_0^{-1}=r_*^{-1}\oplus r_0\oplus r_*^{-1}$, and so $r_0^{-1}\in \cl(\cR)$.
 The proof of Proposition \ref{essn} is complete.
 \end{proof}
 
 In particular, the result characterizes and connects  the property of  $\cR$ to be  essential  with the geometry of $\cl(\cR)$. Considering the isometric involution defined by the inverse $r^{-1}$ of a reaction, we can equivalently say that $\cR$ is essential if and only if $\cl(\cR)$ is symmetric with respect to the above involution.

A semi-linear set is defined as a finite union of linear sets, where  a linear set is a set generated by a base vector $b\in \Z^n$ and \emph{period vectors} $p_1,\ldots p_k\in \Z^n$ as follows \cite{parikh1961language}: 
\[
L(b,p)=\left\{b+\sum_{i=1}^k\lambda_i p_i\Big| \lambda_1,\ldots \lambda_k\in \N_0\right\}.
\]
Semi-linear sets are widely studied in computer science with applications  in automata theory \cite{Parikh_main}, formal languages \cite{ginsburg1964bounded}, and Presburger arithmetic \cite{Seymour/Spanier}, as well as in models of computation, such as Petri nets and vector addition systems \cite{Cook2009}.
 In terms of RNs,  the discrete dynamics of Petri nets and vector addition systems  might equivalently be represented by the discrete dynamics of RNs \cite{Cook2009}. Consequently, the reachable sets of RNs are not semi-linear in general, as this is known  to be the case of Petri nets  and   vector addition systems  \cite{HOPCROFT1979135,normal_petri}.
  Here, we will be concerned with a related question, namely whether the closure  $\cl(\cR)$ of an RN $\cR$, considered as a subset of $\N_0^n\times \N_0^n=\N_0^{2n}$, is semi-linear. 
Simple examples suggest this might be so: if  $\cR=\{\emptyset \ce{<=>} S_1, \emptyset \ce{<=>} S_2, \ldots,  \emptyset \ce{<=>} S_n\}$, then $\cl(\cR)=\N_0^{2n}$; and if $\cR=\{S_1 \ce{<=>} S_2\}$, then $$\cl(\cR)=\{\lambda_1(1,0,0,1)+\lambda_2(0,1,1,0)+\lambda_3(1,0,1,0)+\lambda_4(0,1,0,1) \mid \lambda_1,\lambda_2,\lambda_3,\lambda_4\in\N_0\}.$$
In both cases, the closure is semi-linear, in fact linear. The second example highlights the fact that the closure is generally not  contained in the linear set generated by the reactions.

Following the above discussion, we  ask whether  $\cl(\cR)$ is a semi-linear set. 
This is not the case, as will be seen by example. For this, we need the following lemma.
\begin{lemma}\label{contr_semilin}
Assume $A\subseteq \N_0^n\times\N_0^n$ is semi-linear, and let $x\in \N_0^n$. Then $A_x=\{(a,b)\in A\mid a=x\}$
is empty or a semi-linear set as well.
\end{lemma}

\begin{proof}
It is enough to prove it for $A$ a linear set, as semi-linear sets are finite unions of linear sets. So assume $A$ is linear and that $A_x$ is non-empty. We want to show that $A_x$ is semi-linear. Let $A$ be given by $L(b,p)$ with base vector $b\in \N_0^n\times\N_0^n$ and non-zero period vectors $p_1,\ldots p_k\in \N_0^n\times\N_0^n$. Let $\text{proj}(\cdot)$ denote the projection onto the first $n$-coordinates.

Without loss of generality, we let $p_1,\ldots, p_m$ be the period vectors with non-trivial projection onto the first $n$ coordinates, that is, $\text{proj}(p_i)\neq 0$, $i=1,\ldots,m$ and  $\text{proj}(p_i)= 0$,   $i=m+1,\ldots,k$. 
Then, there are finitely many vectors $\lambda=(\lambda^1,\ldots,\lambda^m) \in \N_0^m$, such that $\text{proj}(b+\sum_{i=1}^m\lambda^ip_i)=x$, as all $p_i$ are non-zero and non-negative. Denote this finite set by $B$, that is, 
$$B=\Big\{b+\sum_{i=1}^m\lambda_ip_i\Big|\lambda \in \N_0^m \text{ and }\text{proj}\Big(b+\sum_{i=1}^m\lambda_ip_i\Big)=x\Big\}.$$
Furthermore, for the first $n$ coordinates, if $p_i$ has a non-zero entry  whenever $x$  in $B$ has a zero entry, then necessarily  $\lambda_i=0$, $i=1,\ldots,m$.

Finally, $A_x$ might be written as 
$$A_x=\bigcup_{c\in B}\Big\{c+\sum_{m+1}^k\lambda_ip_i\Big| \lambda_k,\ldots \lambda_{m+1}\in \N_0\Big\},$$
which is a finite union of linear sets, hence a semi-linear set.
\end{proof}

As it is onerous to prove that a set is not semi-linear, we consider a concrete RN. Using Lemma \ref{contr_semilin}, we reduce the RN  to a known example originally given for a vector addition system, which is not semi-linear  \cite{HOPCROFT1979135}.

\begin{corollary}
There exists an RN $\cR$ such that the closure $\cl(\cR)$ of $\cR$ is not  a semi-linear set.
\end{corollary}

\begin{proof}
We will use that there exists a 6-dimensional vector addition system with a reachability set that is not semi-linear \cite{HOPCROFT1979135}. To conclude we translate that example to the following RN, 
$$S_0+S_2\ce{->} S_0+S_1,\quad S_0\ce{->} S_3,\quad,S_3+S_1\ce{->} S_3+2S_2,\quad S_3\ce{->} S_0+S_4,$$
such that the reachability set of  \cite[Lemma 2.8]{HOPCROFT1979135}, which is  not semi-linear, corresponds to $\cR(S_0+S_2)$ (the reachability set $\cR(x)$ with $x=S_0+S_2$).

We construct a new RN, $\widetilde{\cR}=\{S_5\ce{->} S_0+S_2\}\cup\cR$. Then, $\widetilde{\cR}(S_5)=\cl(\widetilde{\cR})_{S_5}$, where $\cl(\widetilde{\cR})_{S_5}$ is $A_x$ with $A=\cl(\widetilde{\cR})$ and $x=S_5$ (see Lemma \ref{contr_semilin}),  can be written as $\{S_5\}\cup\cR(S_0+S_2)$. \black{We note that the union of a finite set with a non-semi-linear set is  a non-semi-linear set, hence it follows by contradiction and Lemma \ref{contr_semilin} that $\cl(\widetilde{\cR})$ is not semi-linear.}
\end{proof}

\section{Reduction of RNs}\label{sec:reduction}

In this section, we study  graphical  reduction of  an RN to a smaller (reduced) RN in terms of the number of species, entirely based on the reactions alone and not their stochastic propensities to occur. The number of reactions of the reduced RN might be bigger or smaller than the original RN.
Specifically, we provide a definition of  {\it eliminable} species  and that of a {\it reduced} RN, obtained by removal of a set eliminable species.

 The motivation  comes from studying stochastic  RNs with fast-slow dynamics  \cite{CW16,HW2021}. We motivate with an example.

\begin{example}\label{ex:gene}
A simple model of protein production is the following:
\[G\ce{<=>} G',\quad G'\ce{->} G'+P,\quad P\ce{->} 0,\]
where $G$ denotes the inactive state of a gene and  $G'$ the active state, and $P$ is a protein produced while the gene is active \cite{peccoud1995markovian}. The protein is  subsequently  degraded. One might interpret the RN as modelling a single polyploid cell with $K$ copies of the gene, some of which will be in the active state, while  the rest will be in the inactive state. Human cells are diploid and $K=2$.

Assume the reactions involving the active gene in the reactant, $G'\ce{->} G,$ $G'\ce{->} G'+P$, occur at a fast rate compared to the other two reactions. Then it is reasonable to assume that whenever a  gene copy  is activated, a sequence of fast reactions that  eventually ends with deactivation of the gene copy again, occurs before a protein is degraded or another gene copy is activated. Such a sequence (including conversion of $G$ into $G'$) takes the form
$$G\ce{->} G',\quad \black{\underbrace{G'\ce{->} G'+P,\quad\ldots,\quad G'\ce{->} G'+P}_{k\ \mathrm{instances}}},\quad G'\ce{->} G.$$
The net effect of the sequence is simply the sum of the reactions: $G\ce{->} G+kP$. It appears that the active gene $G'$ has been eliminated from the RN through the fast reactions $G'\ce{->} G,$ $G'\ce{->} G'+P$.

To formalize this, let $\cU=\{G'\}$ and $\cF=\{G'\ce{->} G, G'\ce{->} G'+P\}$. Then, we say $\cU$ is eliminable with respect to $\cF$, resulting in the reduced RN,
\[\cR_{\cU,\cF}^*=\{P\ce{->} 0\}\cup\{G\ce{->} G+kP|k\in\N_0\}.\]
The reduced RN has infinitely many reactions.
\end{example}

In the example above,  any sequence of fast reactions  (those of $\cF$) will eventually  be `terminated' by $G'\ce{->} G$. If only $G'\ce{->} G'+P$ is fast, while  $G'\ce{->} G$ is not, then arbitrarily many 
protein copies would be produced before the gene copy is deactivated again. In this case, the reduced RN does not  make sense. Thus, it should  be a requirement that any such sequence of fast reactions is eventually terminated. Oppositely, if only $G'\ce{->} G$ is fast, then the reaction $G'\ce{->} G'+P$ is essentially blocked from occurring as there will be no active gene copies.  \black{Thus, it is reasonable to remove $G\ce{->}G+P$ from the reduced RN.}

To formalize elimination and reduction, we introduce some notation.
Let $\cR\subseteq \N_0^n\times \N_0^n$ be an RN and let $\cU\subseteq\cS$. Furthermore, let $\cR_\cU\subseteq \cR$ and $\cR_\cU'\subseteq \cR$ be the subsets of reactions containing species of $\cU$ in the {\it reactant} and the {\it product}, respectively, 
\begin{align}
\cR_{\cU}&=\{y\ce{->} y'\in \cR\ |\ \cU\cap \supp(y)\neq \emptyset\},\nonumber \\ \cR_{\cU}'&=\{y\ce{->} y'\in \cR\ |\ \cU\cap \supp(y')\neq \emptyset\},\label{cru}
\end{align}
where $\supp(x)=\{S_k | k=1,\dots, n, x^k>0\}$ is the support of $x\in\N_0^n$.
Let $\overline{\cR}=\cl(\cR)$ for convenience, and denote by $\overline{\cR}_{\cU}$ and  $\overline{\cR}_{\cU}'$ the collection of elements in $\overline{\cR}$ containing species of $\cU$ in the reactant and the product, respectively, analogously to \eqref{cru}. 
 We also write $\cR_0=\cR\setminus (\cR_{\cU}\cup \cR_{\cU}')$ and $\overline{\cR}_0=\overline{\cR}\setminus(\overline{\cR}_{\cU}\cup\overline{\cR}_{\cU}')$. It is straightforward to see that $\overline{\cR}_0$ is a closed set (under $\oplus$), $\overline{\cR}_0\supseteq \cl(\cR_0)$,  and in general, $\overline{\cR}_0\neq\cl(\cR_0)$.

We proceed by defining the reduction procedure, and give  further examples below. 

\begin{definition}\label{def.red2}
 Let $\cR$ be an RN and $\cU\subseteq\cS$. The species in $\cU$ are said to be {\it eliminable} (and the set $\cU$ also {\it eliminable}) in $\cR$ with respect to a set of reactions $\cF\subseteq\cR_\cU$, if for any 
$r_0\in\cR_\cU'$ and any $r_1\in \cl(\cF)$ such that $r_0\oplus r_1\not\in \overline{\cR}_{\cU}$, there exists $r_2\in\cl(\cF)$ such that $r_0\oplus r_1\oplus r_2\in \overline{\cR}_{0}$. 

The {\it reduced RN} associated to this elimination is $\cR_{\cU,\cF}^*=\cR_0\cup \cR_{\cU,\cF}$, where
\begin{equation}\label{eq:ruf}
\cR_{\cU,\cF}=\{r_0\oplus r_1\in\overline{\cR}_{0} |\ r_0\in \cR_{\cU}', r_1\in \cl(\cF)  \}\setminus \{r\in\N_0^n\times \N_0^n|r\sim (0,0)\}.
\end{equation}
\end{definition}

Recall that $0\in \cl(\cF)$ (see Section \ref{sec.agb}). As a consequence, if $r_0\oplus r_1\in \overline{\cR}_0$, then it follows that $r_0\oplus r_1\notin \overline{\cR}_{\cU}$. In that case, if we choose $r_2=0\in \cl(\cF)$, then $r_0\oplus r_1\oplus r_2=r_0\oplus r_1\in \overline{\cR}_0$. Thus, when verifying  eliminability of a  subset of species, we only need to consider the case   \black{$r_0\oplus r_1\notin \overline{\cR}_{\cU}\cup \overline{\cR}_0$, that is, $r_0\oplus r_1\in\overline{\cR}_{\cU}'\setminus \overline{\cR}_{\cU}$}.

\black{
We provide the following interpretation of eliminability.  If a set $\cU$  is eliminable, then for any reaction $r_0$ that has species in $\cU$ in its product, but not in its reactant, and any finite sum of reactions from $\cF$, that is, any $r_1\in \cl(\cF)$, the following holds:
\begin{itemize}
\item[-] either $r_0 \oplus r_1$ is in $\overline{\cR}_0$, that is, it does not contain  species in $\cU$ in its reactant nor product,
\item[-] or there is a finite sum of reactions from $\cF$, which  equals $r_2$, and such that $r_0 \oplus r_1 \oplus r_2$ is in $\overline{R}_0$.
\end{itemize}
}

From Definition \ref{def.red2}, it is clear that the reduced RN might be identified as a subset of $\N_0^{n-d}\times \N_0^{n-d}$, with $|\cU|=d\le n$.

\begin{example}
We return to Example \ref{ex:gene} with $\cU=\{G'\}$ and $\cF=\{G'\ce{->} G, G'\ce{->} G'+P\}$. Choosing $r_0,r_1$ as in Definition \ref{def.red2}, such that $r_0\oplus r_1\not\in \overline{\cR}_{\cU}$, then $r_0=G\ce{->} G'$, and $r_1$ is  either \black{the sum of $k$ instances of the reaction $G'\ce{->} G'+P$},   
or \black{the sum of $k$ instances of the reaction $G'\ce{->} G'+P$ with an additional summation by} $G'\ce{->} G$. In the first case, $r_0\oplus r_1=G\ce{->} G'+kP$, and in the second,  $r_0\oplus r_1=G\ce{->} G+kP$. In the latter case, Definition \ref{def.red2} is fulfilled by choosing $r_2=0\in\cl(\cF)$, while in the former, the definition is fulfilled by choosing $r_2=G'\ce{->} G$.
\end{example}

We might interpret Definition \ref{def.red2} in the following way. If $r_0\oplus r_1=(y,y')\not\in \overline{\cR}_{\cU}$, then  $y'$ can be produced from $y$ alone (which contains no species of $\cU$). If $y'$ contains  species of $\cU$, then these can be degraded by reactions of $\cF$. This is guarenteed by the existence of $r_2$. Thus, any  species of $\cU$ produced in this way, can subsequently be degraded again through fast reactions. 

\begin{example}
A more realistic model of protein production is the following \cite{hornos2005self,kepler2001stochasticity}:
\[G\ce{<=>} G',\quad G'\ce{->} G'+R,\quad R\ce{->} R+P,\quad R\ce{->} 0,\quad P\ce{->} 0,\]
where $G$,  $G'$, and $P$ are as before, and $R$ is an intermediate molecule (the mRNA), produced by transcription of the gene. The mRNA is produced by the active gene, and each copy of the mRNA is subsequently translated into protein. Both mRNA and protein might be degraded.

Take $\cU=\{G',R\}$ and 
$$\cF=\cR_\cU=\{G'\ce{->}G, G'\ce{->} G'+R,R\ce{->} R+P,R\ce{->} 0\}.$$
Furthermore, $\cR_{\cU}'=\{G\ce{->}G',G'\ce{->} G'+R,R\ce{->} R+P\}$. If $r_0\in\cR_{\cU}'$ and $r_1\in \cl(\cF)$ are such that $r_0\oplus r_1\not\in \overline{\cR}_{\cU}$, then it must be that  $r_0=\{G\ce{->}G'\}$. Examples of $r_1$, fulfilling the requirement   $r_0\oplus r_1\not\in \overline{\cR}_{\cU}$, include:
\begin{itemize}
\item[(i)] $r_1=G'\ce{->}G$;, here $r_0\oplus r_1=G\ce{->}G$,
\item[(ii)] $r_1=(G'\ce{->} G'+R)\oplus(R\ce{->} R+P)\oplus(R\ce{->} R+P)=G'\ce{->} G'+R+2P$; here $r_0\oplus r_1=G\ce{->}G'+R+2P$,
\item[(iii)] $r_1=(G'\ce{->} G'+R)\oplus(R\ce{->} R+P)\oplus(R\ce{->} 0)=G'\ce{->} G'+P$; here $r_0\oplus r_1=G\ce{->}G'+P$.
\end{itemize}
In either case, there exists $r_2\in\cl(\cF)$, such that $r_0\oplus r_1\oplus r_2\in \overline{\cR}_0$: (i) $r_2=G'\ce{->} G$, (ii) $r_2=(G'\ce{->} G)\oplus (R\ce{->} 0)$, (iii) $r_2=G'\ce{->} G$.

The set $\cU$ is eliminable with respect to $\cF$, resulting in the reduced RN,
\[\cR_{\cU,\cF}^*=\{P\ce{->} 0\}\cup\{G\ce{->} G+kP|k\in\N_0\},\]
which is the same RN as in Example \ref{ex:gene}.
\end{example}

We elaborate further on the properties of elimination.

\begin{lemma}\label{lem:red1}
Let $\cR$ be an RN, and let $\cU\subseteq\cS$. If $r_0\in\cR_\cU'$, $r_1=\oplus_{i=1}^m r_{1i},$ $r_{1i}\in\cR_{\cU}$, $i=1,\ldots,m$, such that $r_0\oplus r_1\not\in \overline{\cR}_{\cU}$,  then $r_0\in\cR_\cU'\setminus\cR_{\cU}$ and $r_0\oplus (\oplus_{i=1}^k r_{1i})\in\overline{\cR}_{\cU}'\setminus \overline{\cR}_{\cU}$ for $k=1,\ldots,m-1$. If $r_0\oplus r_1\in \overline{\cR}_0$, then $r_{1m}\in\cR_\cU\setminus\cR_\cU'$ and $\oplus_{i=k}^{m}r_{1i}\in \overline{\cR}_{\cU}\setminus \overline{\cR}_{\cU}'$ for $k=1,\ldots,m$. 
\end{lemma}

\begin{proof}
Let  $(z_0,z_0')=r_0$ and $(z_k,z_k')=r_0\oplus (\oplus_{i=1}^k r_{1i})$, $k=1,\ldots,m$. It follows from Proposition \ref{properties_sum}\eqref{properties_sum4} that $z_0\le z_1\le\ldots\le z_m$. As $r_0\oplus r_1\not\in\overline{\cR}_{\cU}$, then $\supp (z_k)\cap \cU=\emptyset$ for $k=0,\dots, m$. This implies that $r_0=(z_0,z_0')\notin \cR_{\cU}$, and thus $r_0\in \cR_{\cU}'\setminus\cR_{\cU}$. On the other hand, for  $k=1,\dots, m$,  $r_{1k}=y_{k}\ce{->}y_{k}'\in \cR_{\cU}$ and by definition, $z_{k}=z_{k-1}+0\vee (y_{k}-z_{k-1}')$. 
 As $\supp (z_{k-1})\cap \cU=\supp (z_k)\cap \cU=\emptyset$, we therefore necessarily have 
 $$ \emptyset\neq \supp(y_k)\cap \cU\subseteq \supp(z_{k-1}')\cap \cU,$$
and hence $r_0\oplus (\oplus_{i=1}^k r_{1i})\in\overline{\cR}_{\cU}'\setminus \overline{\cR}_{\cU}$.

 For the second part of the lemma, we have
$r_0\oplus r_1=r_0\oplus (\oplus_{i=1}^{m-1}r_{1i})\oplus r_{1m}$. 
By Proposition \ref{properties_sum}\eqref{properties_sum4}, if $r_{1m}\in \cR_{\cU}'$, then $r_0\oplus r_1\in \overline{\cR}_{\cU}'$, which contradicts the assumption that $r_0\oplus r_1\in \overline{\cR}_0$. Thus, $r_{1m}\in\cR_\cU\setminus\cR_\cU'$. Let  $(\widetilde{z}_k,\widetilde{z}_k')=\oplus_{i=m-k+1}^m r_{1i}$ for all $k=1,\dots, m$, and let $(\widetilde{z}_{m+1},\widetilde{z}_{m+1}')=r_0\oplus (\oplus_{i=1}^m r_{1i})$. Using Proposition \ref{properties_sum}\eqref{properties_sum4} again, we get $\widetilde{z}_1'\le \widetilde{z}_2'\le\ldots\le \widetilde{z}_{m+1}'$. As $r_0\oplus r_1=(\widetilde{z}_{m+1},\widetilde{z}_{m+1}')\in \overline{\cR}_0$, we have $\supp(z_k')\cap \cU=\emptyset$, and thus $\oplus_{i=m-k+1}^m r_{1i}\notin \overline{\cR}_{\cU}'$ for all $k=1,\dots, m$. 
Finally, for any $k\in \{1,\dots,m\}$, as $r_{1k}\in \cR_{\cU}$, Proposition \ref{properties_sum}\eqref{properties_sum4} implies that $(\widetilde{z}_k,\widetilde{z}_k')=\oplus_{i=m-k+1}^m r_{1i}\in \overline{\cR}_{\cU}$. Therefore, $\oplus_{i=m-k+1}^m r_{1i}\in \overline{\cR}_{\cU}\setminus \overline{\cR}_{\cU}'$ for all $k=1,\dots, m$. It completes the proof.
\end{proof}

If $r\in\overline{\cR}$ is such that $x\in\N_0^n$ is active on $r$ and $\cU\cap \supp(x)=\emptyset$, then $r\not\in \overline{\cR}_{\cU}$ (cf. Corollary \ref{coroactive}). In particular, this applies to reactions $r$ that appear as sums of reactions $r=\oplus_{i=0}^m r_{i}$ with $r_0\in\cR_\cU'$ and $r_1,\ldots, r_m\in\cF$.

Keeping the interpretation of fast-slow dynamics in mind, let $\cF$ consist of the fast reactions and $\cR\setminus\cF$ of the slow reactions. If currently in a state $x\in\N_0^n$ with no molecules of the species in $\cU$, that is,  $\cU\cap \supp(x)=\emptyset$, and a reaction $r_0\in\cR_\cU'\setminus \cR_{\cU}$ occurs, producing one or more molecules of the species in $\cU$, then usually
a sequence of reactions takes place that degrades the molecules of the species $\cU$ again. Reactions in $\cR_\cU\setminus\cF$ have a low probability of occuring \cite{HW2021}.

We state some trivial cases of eliminable species.
\begin{enumerate}[(i)]
\setlength\itemsep{1em}
\item If $\cU\subseteq \cS$ and $\cR_\cU'\setminus \cR_{\cU}= \emptyset$, then $\cU$ is eliminable with respect to any $\cF\subseteq \cR_{\cU}$. In that case $\cR_{\cU,\cF}=\emptyset$ and $\cR_{\cU,\cF}^*=\cR_0$ (cf. Lemma \ref{lem:red1}).
\item If $\cU=\emptyset$, then $\cR_{\cU}=\cR_{\cU}'=\emptyset$, and $\cU$ is eliminable with respect to $\cF=\emptyset$,  and $\cR_0=\cR$, $\cR_{\cU,\cF}=\emptyset$ and thus $\cR_{\cU,\cF}^*=\cR$.
\item If $\cU=\cS$, then $\cR_0=\emptyset$, $\cR_{\cU}=\cR_{\cU}'=\cR$ and hence this is a special case of (i) with $\cR_{\cU,\cF}^*=\emptyset$.
\end{enumerate}

If $\cU$  is eliminable with respect to both $\cF_1\subseteq\cR_{\cU}$ and  $\cF_2\subseteq\cR_{\cU}$, then $\cU$ is not necessarily eliminable with respect to $\cF_1\cup \cF_2$. The same is the case if disjoint sets $\cU_1$ and  $\cU_2$ are eliminable with respect to $\cF_1$ and  $\cF_2$, respectively (potentially with empty intersection), then $\cU_1\cup\cU_2$ is not necessarily eliminable with respect to $\cF_1\cup\cF_2$. Here, it should at least be required that $\cU_2$ is eliminable with respect to $\cF_2\subseteq(\cR^*_{\cU_1,\cF_1})_{\cU_2}$, see Proposition \ref{prop:ny}. An example of this is also given below, where the two sets of eliminable species do not appear in the same reactions, hence the condition is trivially fulfilled.

\begin{proposition}
Let $\cR$ be an RN, and let $\cU=\cU_1\cup \cU_2\subseteq\cS$ with $\cU_1\cap \cU_2=\emptyset$. Suppose that $\cU_1$ in $\cR$ is eliminable with respect to $\cF_1$, that $\cU_2$ is eliminable in $\cR$ with respect to $\cF_2$, and that 
\[(\cR_{\cU_1}\cup \cR_{\cU_1}')\cap (\cR_{\cU_2}\cup \cR_{\cU_2}')=\emptyset.\]
Then, $\cU$ is eliminable in $\cR$ with respect to 
$\cF=\cF_1\cup \cF_2$.
\end{proposition}

\begin{proof}
Let $r_0\oplus r_1\not\in\overline{\cR}_\cU$ with $r_0\in\cR_\cU'=\cR_{\cU_1}'\cup\cR_{\cU_2}'$, $r_1\in\cl(\cF)$.  If $r_0\oplus r_1\notin \overline{\cR}_\cU'$, then we are done.
Otherwise, suppose that $r_0\oplus r_1\in\overline{\cR}_\cU'$.  Without loss of generality, assume that $r_0\in\cR_{\cU_1}'$. Let $r_1=\oplus_{i=1}^mr_{1i}$ with $r_{1i}\in\cF$. As $r_0\oplus r_1\notin \overline{\cR}_{\cU}$, then by Lemma \ref{lem:red1},  we have  $r_0\in\cR_{\cU_1}'\setminus \cR_{\cU}$ and $r_0\oplus (\oplus_{i=1}^k r_{1i})\in\overline{\cR}_{\cU}'\setminus \overline{\cR}_{\cU}$ for all $k=1,\dots, m$. Note that $\cR_{\cU_1}'\cap \cR_{\cU_2}'=\emptyset$, hence   $r_0$ has no species of $\cU_2$ in the product. We claim that $r_{11}$ has no species of $\cU_2$ in the reactant. \black{If this is not the case, then $r_0\oplus r_{11}\in \overline{\cR}_{\cU_2}\subseteq  \overline{\cR}_{\cU}$, which contradicts the fact that $r_0\oplus r_{11}\in \overline{\cR}_{\cU}'\setminus \overline{\cR}_{\cU}$.} Thus, $r_{11}\in \cF_1$. 

Recall the assumption that $\cR_{\cU_1}\cap \cR_{\cU_2}'=\emptyset$. It follows that $r_{11}\in \cR_{\cU_1}\setminus \cR_{\cU_2}'$, and thus, by Proposition \ref{properties_sum}\eqref{properties_sum4}, $r_0\oplus r_{11}\in \overline{\cR}_{\cU_1}'\setminus\overline{\cR}_{\cU_2}'$. As a result, $r_{12}$ has no species of $\cU_2$ in the reactant as well. This implies that $r_{12}\in \cF_1$. Iteratively, we can show that $r_{1k}\in \cF_1$ for all $k=1,\dots, m$. In other words, $r_1=\oplus_{k=1}^mr_{1k}\in \cl(\cF_1)$. Since $\cU_1$ is eliminable with respect to $\cF_1$, there exists $r_2\in \cl(\cF_1)\subseteq \cl(\cF)$ such that $r_0\oplus r_1\oplus r_2\in \overline{\cR}\setminus(\overline{\cR}_{\cU_1}\cup\overline{\cR}_{\cU_1}')$. By assumption $(\cR_{\cU_1}\cup \cR_{\cU_1}')\cap (\cR_{\cU_2}\cup \cR_{\cU_2}')=\emptyset$, we have $r_0\notin \cR_{\cU_2}\cup \cR_{\cU_2}'$ and $\cl(\cF_1)\cap (\overline{\cR}_{\cU_2}\cup \overline{\cR}_{\cU_2}')=\emptyset$. Thus $r_0\oplus r_1\oplus r_2\notin (\overline{\cR}_{\cU_2}\cup \overline{\cR}_{\cU_2}')$, which yields that 
\[
r_0\oplus r_1\oplus r_2\in \overline{\cR}_0= \overline{\cR}\setminus ( \overline{\cR}_{\cU}\cup \overline{\cR}_{\cU}')= \overline{\cR}\setminus ( \overline{\cR}_{\cU_1}\cup \overline{\cR}_{\cU_1}'\cup\overline{\cR}_{\cU_2}\cup \overline{\cR}_{\cU_2}').
\]
This completes the proof of this lemma.
\end{proof}

We introduce some important classes of species that often appear in practice \cite{Variable_el,saez,saez2}. See also Example \ref{ex_2MMrev}.
 
\begin{definition}
Let $\cR$ be an RN and $\cU\subseteq \cS$.  Then,
\begin{enumerate}[(i)]
\setlength\itemsep{1em}
\item  $\cU$ consists of {\it non-interacting} species, if for any two species $S_i,S_j\in\cU$ and any reaction $y\ce{->}y'\in\cR$, the sum of the stoichiometric coefficients $y^i+y^j$ and $(y')^i+(y')^j$ in the reactant and the product, respectively, are at most one.
\item 
$\cU$ consists of {\it intermediate} species, if the species of $\cU$ are non-interacting and furthermore, for $S_i\in\cU$ and $y\ce{->}y'\in\cR$, whenever $y^i=1$, then $y=S_i$, and whenever  $(y')^i=1$, then $y'=S_i$.
\end{enumerate} 
\end{definition}

\begin{example}[Example \ref{ex_2MM} revisited]\label{ex_2MMrev}
Recall the reactions
$$E+A\ce{<=>} EA,\quad EA+P\ce{->}EQ\ce{->}E+Q.$$
The set $\cU=\{EQ\}$ consists of intermediate species and $\cU$ is eliminable with respect to $\cF=\{EQ \ce{->} E+Q\}$. The reduced RN  is $\cR_{\cU,\cF}^*=\{E+A\ce{<=>} EA, EA+P\ce{->}E+Q\}$.
Similarly, the set $\cU=\{EA, EQ\}$ consists of non-interacting species and $\cU$ is eliminable with respect to $\cF=\{EA\ce{->} E+A, EA+P\ce{->}EQ \ce{->} E+Q\}$. The reduced RN  is $\cR_{\cU,\cF}^*=\{ E+A+P\ce{->}E+Q\}$.
\end{example}

\begin{lemma} \label{lem:non-int-red}
Let $\cR$ be an RN and  $\cU\subseteq\cS$ a set of non-interacting species.
Furthermore, let $r_0\in\cR_\cU'$, $r_1=\oplus_{i=1}^m r_{1i},$ $r_{1i}\in\cR_\cU$, $i=1,\ldots,m$, such that $r_0\oplus r_1\notin \overline{\cR}_{\cU}$. Then,
\begin{enumerate}[(i)]
\setlength\itemsep{1em}
    \item \label{non-int-red1} $r_0\in\cR_\cU'\setminus\cR_\cU$, $r_{1i}\in\cR_\cU\cap\cR_\cU'$, $i=1,\ldots,m-1$. 
    \item \label{non-int-red2}
     Assume $r_0=y_0\ce{->} y_0'$ and $r_{1i}=y_i\ce{->} y_i'$, $i=1,\dots, m$. Then, $\supp(y_{i})\cap\cU=\supp(y_{i-1}')\cap \cU\neq \emptyset$ for $i=1,\ldots ,m$. 
    \item \label{non-int-red3} If $r_0\oplus r_1\in \overline{\cR}_0$, then $r_{1m}\in\cR_\cU\setminus\cR_\cU'$.
\end{enumerate}
Oppositely, let $r_0\in\cR_\cU'$, $r_1=\oplus_{i=1}^m r_{1i},$ $r_{1i}\in\cR_{\cU}$, $i=1,\ldots,m$. Suppose that   both \eqref{non-int-red1} and \eqref{non-int-red2} hold. Then,  $r_0\oplus r_1\notin \overline{\cR}_{\cU}$. If furthermore, $r_{1m}\in \cR_{\cU}\setminus\cR_{\cU}'$, then $r_0\oplus r_1\in \overline{\cR}_0$.
\end{lemma}

 \begin{proof}
The backward direction of the lemma is straightforward, so we only need to prove the forward direction. Note that Lemma \ref{lem:red1} and \eqref{non-int-red2} imply \eqref{non-int-red1} and \eqref{non-int-red3}, hence we are left to prove \eqref{non-int-red2}.  Assume $m\geq 2$, as otherwise there is nothing to prove. Recall that $\cU$ is a set of non-interacting species. Since $r_k\in \cF\subseteq\cR_{\cU}$, $k=1,\dots,m$, then each reactant $y_k$ contains exactly one species in $\cU$ with stoichiometric coefficient one.  Let $(z_k,z_k')=r_0\oplus (\oplus_{i=1}^k r_i)$, $k=1,\dots, m$. By repeating the proof in Lemma \ref{lem:red1}, we find  that $\supp(y_1)\cap \cU=\supp(y_0')\cap \cU$ and $\supp(y_0'-y_1)\cap \cU=\emptyset$. Thus, $\supp(z_1')\cap\cU=\supp(y_1'+0\vee(y_0'-y_1))\cap \cU=\supp(y_1')\cap \cU$ is a singleton. Therefore, $\supp(y_1')\cap\cU=\supp(z_1')\cap\cU=\supp(y_2)\cap \cU$. The proof of this lemma can be completed by iteration.
\end{proof}

The conclusion of Lemma \ref{lem:non-int-red} is not true in {\it general}, see  Example \ref{exam4}.

\begin{proposition}\label{prop:ny}
Let $\cR$ be an RN and let $\cU=\cU_1\cup\cU_2\subseteq\cS$
be a set of non-interacting species such that $\cU_1\cap\cU_2=\emptyset$. Furthermore, assume $\cU_1$ is eliminable with respect to $\cF_1\subseteq\cR_{\cU_1}$ in $\cR$, and that $\cU_2$ is  eliminable  with respect to $\cF_2=(\cR^*_{\cU_1,\cF_1})_{\cU_2}$ in $\cR^*_{\cU_1,\cF_1}$. Then $\cU$ is eliminable with respect to $\cF=\cF_1\cup\cR_{\cU_2}$   in $\cR$.
\end{proposition}

\begin{proof}
We make use of the following notation: $\widetilde{\cR}=\cR^*_{\cU_1,\cF_1}$, $\widehat{\cR}=\cl(\widetilde{\cR})$, $\widehat{\cR}_0=\widehat{\cR}\setminus (\widehat{\cR}_{\cU_2}\cup \widehat{\cR}_{\cU_2}')$, and $\overline{\cR}_0=\overline{\cR}\setminus (\overline{\cR}_{\cU}\cup \overline{\cR}_{\cU}')$. 
 \black{ Then, $\widehat{\cR}$ is a closed subset of $\overline{\cR}$ such that for any $(y,y')\in \widehat{\cR}$, $\supp(y)\cap \cU_1=\supp(y')\cap \cU_1=\emptyset$. Thus, $\widehat{\cR}\subseteq \overline{\cR}\setminus(\overline{\cR}_{\cU_1}\cup \overline{\cR}_{\cU_1}')$. Furthermore, by definition $\widehat{\cR}_0$ consists of all $(y,y')\in\widehat{\cR}$ such that $\supp(y)\cap \cU_2=\supp(y')\cap \cU_2=\emptyset$, thus  $\widehat{\cR}_0\cap (\overline{\cR}_{\cU_2}\cup \overline{\cR}_{\cU_2}')=\emptyset$.} It follows that
\[
\widehat{\cR}_0\subseteq \overline{\cR}\setminus (\overline{\cR}_{\cU_1}\cup\overline{\cR}_{\cU_1}'\cup\overline{\cR}_{\cU_2}\cup\overline{\cR}_{\cU_2}')=\overline{\cR}_0.
\]
By definition, we need to verify that for any $r_0\in \cR_{\cU}'$, $r_1=\oplus_{i=1}^m r_{1i}$, $r_{1i}\in \cF$, with $r_0\oplus r_1\notin \overline{\cR}_{\cU}$, either $r_0\oplus r_1\in \overline{\cR}_0$, or there exists $r_2\in \cl(\cF)$ such that 
\begin{align}\label{r2}
r_0\oplus r_1\oplus r_2\in \overline{\cR}_0. 
\end{align}
Before proving this property, we show  that $\cl(\cF_1)\subseteq \cl(\cF)$ and $\cl(\cF_2)\subseteq \cl(\cF)$. The first inclusion is trivial. Now we prove the second one. Let $r\in \cF_2$. 
Then, by definition either $r\in \cF_2\cap [\cR\setminus (\cR_{\cU_1}\cup \cR_{\cU_1}')]$ or $r\in \cF_2\cap \cR_{\cU_1,\cF_1}$. In the former case, we have $r\in \cF_2\cap \cR_{\cU_2}\subseteq \cF$. Thus, it suffice to consider the second case, for which, we can write $r=\widehat{r}_0\oplus \widehat{r}_1\in \cR_{\cU_1,\cF_1}$ with $\widehat{r}_0\in \cR_{\cU_1}'$ and $\widehat{r}_1=\oplus_{k=1}^m\widehat{r}_{1k}$, $\widehat{r}_{1k}=\widehat{y}_{1k}\ce{->}\widehat{y}_{1k}'\in \cF_1\subseteq \cR_{\cU_1}$. Thus, $\supp(y_{1k})\cap \cU_1\neq\emptyset$, and by the non-interacting property, it holds that $\supp(y_k)\cap \cU_2=\emptyset$ for all $k=1,\dots, m$. 
Therefore, due to Proposition \ref{properties_sum}\eqref{properties_sum4}, we have  $\widehat{r}_1\in \overline{\cR}_{\cU_1}\setminus \overline{\cR}_{\cU_2}$. Hence, by Proposition \ref{properties_sum}\eqref{properties_sum4} again and because $\widehat{r}_0\oplus \widehat{r}_1\in \cF_2\subseteq \overline{\cR}_{\cU_2}$ has a non-interacting species in $\cU_2$ in the reactant, $\widehat{r}_0\in \cR_{\cU_2}\subseteq \cF$. Thus, $r\in \cl(\cF)$.

Suppose that $r_0\oplus r_1\in \overline{\cR}_{\cU}'\setminus \overline{\cR}_{\cU}$. Next, we show the existence of  an $r_2\in \cl(\cF)$, such that \eqref{r2} holds. Recall $r_1=\oplus_{i=1}^m r_{1i}$, $r_{1i}\in \cF$ such that $r_0\oplus r_1=r_0\oplus(\oplus_{i=1}^m r_{1i})\in \overline{\cR}_{\cU}'\setminus \overline{\cR}_{\cU}$. We claim that $r_{1m}\in \cR_{\cU}'$. Otherwise, assume $r_{1m}\in \cR_{\cU}\setminus\cR_{\cU}'$. By the opposite part of Lemma \ref{lem:non-int-red}, we have $r_0\oplus r_1\in \overline{\cR}_0$, which  contradicts the assumption that $r_0\oplus r_1\in \overline{\cR}_{\cU}'\setminus \overline{\cR}_{\cU}$. Thus, we have  $r_{1m}\in \cR_{\cU}\cap \cR_{\cU}'$. Suppose that $r_{1m}\in \cR_{\cU_1}\cap \cR_{\cU_1}'$. Let $j$ be the largest index  strictly smaller than $m$ such that $r_{1j}\notin \cR_{\cU_1}$ (with $r_{10}=r_0$). If $j=0$, then $r_{1i}\in (\cR_{\cU_1}\cap\cR_{\cU_1}')\cap \cF\subseteq \cF_1$, $i=1,\dots, m$ and, by  Lemma \ref{lem:red1}, $r_0\in \cR_{\cU_1}'\setminus \cR_{\cU_1}$. As $\cU_1$ is eliminable with respect to $\cF_1$, there exists $\widehat{r}_2\in \cl(\cF_1)$ such that $\widehat{r}=r_0\oplus r_1\oplus \widehat{r}_2\in \widetilde{\cR}$. If  $\widehat{r}\notin \widetilde{\cR}_{\cU_2}'$, then \eqref{r2} holds with  $r_2=\widehat{r}_2$. Otherwise,  $\widehat{r}\in \widetilde{\cR}_{\cU_2}'\setminus \widetilde{\cR}_{\cU_2}$. Since $\cU_2$ is eliminable in $\widetilde{\cR}$, with respect to $\cF_2$, there exists $\widehat{r}_3\in \cl(\cF_2)\subseteq \cl(\cF)$ such that $r_0\oplus r_1\oplus \widehat{r}_2\oplus \widehat{r}_3\in \widehat{\cR}_0$. Thus, we get \eqref{r2} with $r_2=\widehat{r}_2\oplus \widehat{r}_3$.

On the other hand, if $j>0$, then $r_{1j}\in \cR_{\cU_2}\cap \cR_{\cU_1}'$ and $r_{1(j+1)},\dots, r_{1m}\in \cF_1$. Thus, by  eliminability of $\cU_1$ in $\cR$ with respect to   $\cF_1$, 
there exists $\widehat{r}_2\in \cl(\cF_1)$ such that $r_{1j}\oplus \cdots \oplus r_{1m}\oplus \widehat{r}_2\in \widetilde{\cR}_{\cU_2}$. Let $j'$ be the largest index strictly smaller than $j$ such that $r_{1j'}\notin \cR_{\cU_2}$. If $j'=0$, then $r_0\in \cR_{\cU_2}'\setminus \cR_{\cU}\subseteq \widetilde{\cR}_{\cU_2}'\setminus \widetilde{\cR}_{\cU_2}$. Otherwise, $r_{1j'}\in \cR_{\cU_1}\cap\cR_{\cU_2}'$ and $r_{1(j'+1)},\dots, r_{1j}\in \cR_{\cU_2}\cap \cR_{\cU'_2}\subseteq \widetilde{\cR}_{\cU_2}=\cF_2$. Let $j''$ be the largest index strictly smaller than $j'$ such that $r_{1j''}\notin \cR_{\cU_1}$. By using the opposite part of Lemma \ref{lem:non-int-red}, we see that  $r_{1j''}\oplus \cdots \oplus r_{1j'}\in \widetilde{\cR}_{\cU_2}$. By repeating the same argument, we find that $r_0,r_{11},\cdots, r_{1(j''-1)}$ can be divided into ordered groups such that the sum of the reactions in each group, 
except the first group, is either in $(\cR_{\cU_1,\cF_1})_{\cU_2}$ or $\cR_{\cU_2}\cap\cR_{\cU_2}'$, which are both in $\widetilde{\cR}_{\cU_2}$, and the sum of the reactions in the first group is in $\widetilde{\cR}_{\cU_2}'\setminus\widetilde{\cR}_{\cU_2}$. Hence the existence of $r_2$ follows from  eliminability of $\cU_2$ with respect to $\cF_2=\widetilde{\cR}_{\cU_2}$ in $\widetilde{\cR}$.

The other cases when $r_{1m}$ is in $\cR_{\cU_1}\cap \cR_{\cU_2}'$, $\cR_{\cU_2}\cap \cR_{\cU_1}'$ or $ \cR_{\cU_2}\cap \cR_{\cU_2}'$ are essentially proved by the same means as above. The proof of the proposition is complete.
\end{proof}

It is not sufficient that $\cU_2$ is eliminable with respect to $\cR_{\cU_2}$ in $\cR$. For example, consider the RN $\cR=\{S_1\ce{->} U_1\ce{<=>}U_2\}$. Then $\cU_1=\{U_1\}$ is eliminable with respect to $\cR_{\cU_1}$, and $\cU_2=\{U_2\}$ is eliminable with respect to $\cR_{\cU_2}$. However,  $\cU_1\cup\cU_2$ is not eliminable with respect to $\cR_{\cU_1\cup\cU_2}$.

We further note that several approaches to reductions of RNs derived from a deterministic dynamical perspective, have been studied both in terms of slow-fast dynamics \cite{king-altman,feliu2019quasisteady} as well as in the context of steady-states \cite{plos-12-gunawardena, intermediates,Variable_el,saez,springer-14-pantea-gupta-rawlings-craciun} for intermediates and non-interacting species in general. In the case of intermediates, our reduced RN agrees with the one suggested in \cite{intermediates}. However, for non-interacting species the reduced RN we obtain differs from that of \cite{feliu2019quasisteady,saez}. This is a consequence of the discrete nature of the state space in our case compared to the continuous state space for deterministic reaction systems.

\section{Reversibility analysis for reduced RNs}\label{sec:red_reach}

Reversibility (weak reversibility, essentiality) is an important property for an  RN  and often imply strong properties on the dynamics, irrespectively whether the RN is modelled deterministically or stochastically \cite{anderson2,Cappelletti,Feinberg,anderson1,non-stand_1}.
Therefore, we are interested in finding criteria for a reduced RN to be reversible (weakly reversible, essential), provided the original RN is. However, in general, this appears to be a  challenging problem. Here, we provide  sufficient conditions for a reduced RN to be  (weakly) reversible under the assumption that the eliminable species are non-interacting species.

For a set $A\subseteq\N_0^n\times \N_0^n$, let $A^{-1}=\{r^{-1}|r\in A\}$.

\begin{theorem}\label{thm.rvs}
Let $\cR$ be an RN and $\cU\subseteq\cS$ a set  of non-interacting species. Assume $\cU$ is eliminable with respect to $\cF\subseteq \cR_\cU$, as in Definition \ref{def.red2}, and define the condition
$$(*)\qquad (\cR_{\cU}'\setminus \cR_{\cU})^{-1}= \cF\setminus \cR_{\cU}'\quad \text{and}\quad\cF\cap  \cR_{\cU}'\text{\,\, is essential.}$$
Then,
\begin{enumerate}[(i)]
\setlength\itemsep{1em}
\item \label{rvred0} If  $(\cR_\cU'\setminus \cR_\cU)\cup\cF$ is reversible then $(*)$ holds.
\item \label{rvred1} $\cR_{\cU,\cF}$ is reversible if $(*)$ holds. 
\item \label{rvred2} $\cR_{\cU,\cF}^*$ is (weakly) reversible if $\cR_0$ is (weakly) reversible and $(*)$ holds.
\item \label{rvred3}  $\cR_{\cU,\cF}^*$ is weakly reversible if there exists $\cF_0\subseteq \cl(\cR)$ such that $(\cR\setminus \cR_{\cU})\cup \cF_0$ is weakly reversible and $(*)$ holds.
\end{enumerate}
\end{theorem}

\begin{proof}
 \eqref{rvred0} Firstly, note that $(\cR_\cU'\setminus \cR_\cU)\cup\cF$ can be decomposed into three disjoint sets $\cR_{\cU}'\setminus \cR_{\cU}$, $\cF\setminus \cR_\cU'\subseteq \cR_{\cU}\setminus\cR_{\cU}'$ and $\cF\cap\cR_\cU'\subseteq\cR_{\cU}\cap\cR_{\cU}'$. \black{Since for any $(y,y')\in \cR_{\cU}'\setminus \cR_{\cU}$, $\supp(y)\cap \cU=\emptyset$ and $\supp(y')\cap \cU\neq \emptyset$, it follows that $(y',y)\notin \cR_{\cU}'\setminus \cR_{\cU}$. For the same reason, $(y',y)\notin \cF\cap\cR_\cU'$, where $\cF\cap\cR_{\cU}'\subseteq \cR_{\cU}\cap \cR_{\cU}'$. Thus we have,}
 \begin{align}\label{equ_its0}
 (\cR_{\cU}'\setminus \cR_{\cU})^{-1}\cap (\cR_{\cU}'\setminus \cR_{\cU})=(\cR_{\cU}'\setminus \cR_{\cU})^{-1}\cap(\cF\cap\cR_\cU')=\emptyset.
 \end{align}
 \black{By   reversibility of $(\cR_\cU'\setminus \cR_\cU)\cup\cF$, $(\cR_{\cU}'\setminus \cR_{\cU})^{-1}\subseteq (\cR_\cU'\setminus \cR_\cU)\cup\cF=(\cR_{\cU}'\setminus \cR_{\cU})\cup (\cF\setminus \cR_\cU')\cup(\cF\cap\cR_\cU')$. Combining this fact with \eqref{equ_its0}, we have} $(\cR_{\cU}'\setminus \cR_{\cU})^{-1}\subseteq \cF\setminus \cR_\cU'$. Similarly, it holds that $(\cF\setminus \cR_\cU')^{-1}\subseteq\cR_{\cU}'\setminus \cR_{\cU}$, \black{which, together with $(\cR_{\cU}'\setminus \cR_{\cU})^{-1}\subseteq \cF\setminus \cR_\cU'$,  implies $(\cR_{\cU}'\setminus \cR_{\cU})^{-1}= \cF\setminus \cR_{\cU}'$. For the same reason, we can show that $(\cF\cap\cR_\cU')^{-1}\subseteq \cF\cap\cR_\cU'$ holds. Hence $\cF\cap\cR_\cU'$ is essential}. In other words, $(*)$ is true  and the proof is complete.

 \eqref{rvred1} Let $r_0\oplus r_1\in \cR_{\cU,\cF}$, where $r_0,r_1$ are as in Definition \ref{def.red2}, Eqn. \eqref{eq:ruf}. Furthermore, there exists $r_{11},\dots, r_{1m}\in \cF$, such that $r_1=\oplus_{i=1}^mr_{1i}$. By Lemma \ref{lem:non-int-red},  $r_0\in \cR_{\cU}'\setminus \cR_{\cU}$, $r_{1m}\in (\cR_{\cU}\setminus \cR_{\cU}')\cap \cF=\cF\setminus \cR_{\cU}'$ and $\{r_{11},\dots,r_{1(m-1)}\}\subseteq \cF\cap \cR_{\cU}'$, assuming $m\ge 2$. Therefore, under condition $(*)$, we know that $r_0^{-1}\in \cF\cap (\cR_{\cU}\setminus \cR_{\cU}')$, $r_0' \black{:=}   r_{1m}^{-1}\in \cR_{\cU}'\setminus \cR_{\cU}$, and 
\[
\left(\oplus_{i=1}^{m-1}r_{1i}\right)^{-1}\in \ \cl(\cF\cap  \cR_{\cU}')\subseteq \cl(\cF).
\]
Therefore, $r_1'\black{:=}(\oplus_{i=1}^{m-1}r_{1i})^{-1}\oplus r_0^{-1}\in \cl(\cF)$ and thus $(r_0\oplus r_1)^{-1}=r_1'\oplus r_0'\in \cR_{\cU,\cF}$. This proves property \eqref{rvred1}.

 \eqref{rvred2} It  is a direct consequence of \eqref{rvred1} and the definition of $\cR_{\cU,\cF}^*$.

\eqref{rvred3}  It suffices to show that every $r=y\ce{->} y'\in \cR_0$ is weakly reversible in  $\cR_{\cU,\cF}^*$. Note that  $\cR_0\subseteq \cR\setminus \cR_{\cU}\subseteq (\cR\setminus \cR_{\cU}) \cup \cF_0$.  Thus, by assumption, there exist reactions 
$y'\ce{->} y_1,y_1\ce{->} y_2,\dots,$ $ y_m\ce{->} y\in (\cR\setminus \cR_{\cU})\cup \cF_0$. If for $k=1,\ldots,m$, $\supp(y_k)\cap \cU=\emptyset$, then $r$ is weakly reversible in $\cR_0$ and thus in  $\cR_{\cU,\cF}^*$. Otherwise, let $i=\min\{k|\supp(y_k)\cap \cU\neq \emptyset)\}$. Then 
\[
\{y'\ce{->} y_1,y_1\ce{->} y_2,\dots, y_{i-2}\ce{->} y_{i-1}\}\subseteq \cR_0\subseteq \cR_{\cU,\cF}^*,
\]
and $y_{i-1}\ce{->} y_i\in \cR_{\cU}'\setminus \cR_{\cU}$ (with $y_0=y'$). Let $j=\min \{k>i|\ \supp(y_k)\cap \cU= \emptyset)\}$. Then, 
\[
 \{y_{i}\ce{->} y_{i+1},\dots, y_{j-1}\ce{->} y_j\}\subseteq \cF_0.
\]
Therefore, $(y_i, y_j)=\oplus_{\ell=i}^{j}(y_{\ell-1}\ce{->} y_{\ell})\in \cl(\cF)$, which implies  either $(y_{i-1},y_j)=y_{i-1}\ce{->} y_i\oplus (y_i, y_j)\in \cR_{\cU,\cF}$ or $\sim (0,0)$, see Lemma \ref{lem:non-int-red}. Repeating this process, we can find a sequence of reactions $r_1',\dots, r_p'$ in the reduced RN $\cR_{\cU,\cF}^*$ (after removing elements equivalent to $(0,0)$) such that the product of $r_k'$ coincides with the reactant of $r_{k+1}'$ for  $k=1,\dots, p-1$, and $\oplus_{k=1}^p r_k'=y'\ce{->} y$. The proof of property \eqref{rvred3} is complete.
\end{proof}

We present some examples that show the limitations of Theorem \ref{thm.rvs}.

\begin{example}\label{ex5.2}
Consider the RN 
\[\cR=\{S_1\ce{->} U_1, U_1\ce{->} S_2, S_2\ce{->} S_1\}\]
with $\cU=\{U_1\}$. Let $\cF=\{U_1\ce{->} S_2\}$. Then, the reduced network $\cR_{\cU,\cF}^*=\{S_1\ce{<=>} S_2\}$ is reversible. However, \begin{enumerate}[(i)]
\setlength\itemsep{1em}
\item $\cR_0=\{S_2\ce{->}S_3+S_4\}$ is not reversible,
\item $(\cR_{\cU}'\setminus \cR_{\cU})^{-1}=\{U_1\ce{->} S_1\}\neq \cF\setminus \cR_{\cU}'=\{U_1\ce{->}S_2\}$.
\end{enumerate}
\end{example}

\begin{example}\label{ex5.3}
Concerning Theorem \eqref{thm.rvs}\eqref{rvred3}, consider the RN 
\[\cR=\{S_1+S_2 \ce{->} S_3+S_4, S_3\ce{->} U_1, S_4+U_1\ce{->} S_1+U_2, U_2\ce{->} S_2\}\]
with $\cU=\{U_1,U_2\}$. Let $\cF=\{S_4+U_1\ce{->} S_1+U_2, U_2\ce{->} S_2\}$. Then, the reduced network $\cR_{\cU,\cF}^*=\{S_1+S_2\ce{<=>} S_3+S_4\}$ is reversible. However, 
\begin{enumerate}[(i)]
\setlength\itemsep{1em}
\item $\cR_0=\{S_1+S_2\ce{->}S_3+S_4\}$ is not reversible.
\item $(\cR_{\cU}'\setminus \cR_{\cU})^{-1}=\{U_1\ce{->} S_3\}\neq \cF\cap (\cR_{\cU}\setminus \cR_{\cU}')=\{U_2\ce{->}S_2\}$.
\item There does not exist a subset $\cF_0\subseteq \cl(\cF)$ such that $(\cR\setminus \cR_{\cU})\cup\cF_0$ is essential, because $U_1\ce{->} S_3\in (\cR\setminus \cR_{\cU})^{-1}$ cannot be represented as a sum of reactions in $(\cR\setminus \cR_{\cU})\cup\cl(\cF)$.
\end{enumerate}
\end{example}

Therefore, Example \ref{ex5.2} and Example \ref{ex5.3} imply that the conditions provided in Theorem \ref{thm.rvs} are not necessary conditions for (weakly) reversibility of the reduced RN. The next example shows that  weak reversibility of $(\cR_{\cU}'\setminus \cR_{\cU})\cup \cF$ in the case of non-interacting species does not ensure weak reversibility of the reduced network, 
 implying reversibility in Theorem \ref{thm.rvs}\eqref{rvred0} cannot be replaced by weak reversibility and assumption $(*)$ cannot be removed in Theorem \ref{thm.rvs}\eqref{rvred3}.

\begin{example}
Consider the RN
\[
\cR=\{S_1\ce{->} U_1\ce{->} S_2\ce{->} U_2\ce{->} S_1, S_3+U_2\ce{<=>} S_4\},
\]
with $\cU=\{U_1,U_2\}$, $\cF=\cR_\cU$. 
Then, $(\cR\setminus \cR_{\cU})\cup \cF =(\cR_{\cU}'\setminus \cR_{\cU})\cup \cF=\cR$ is weakly reversible, but $\cR_{\cU,\cF}^*=\{S_1\ce{<=>} S_2, S_2+S_3\ce{->} S_4\ce{->} S_3+S_1\}$ is not weakly reversible. 
\end{example}

The example  below shows that Theorem \ref{thm.rvs} is not true beyond non-interacting species.

\begin{example}\label{exam4}
Consider the RN given by
\[\cR=\{S_1\ce{<=>} U_1+U_2, S_2\ce{<=>} U_1, S_3\ce{<=>} U_2\}\]
with $\cU=\{U_1,U_2\}$, and let $\cF=\cR_{\cU}$. Then,  $(\cR\setminus \cR_{\cU})\cup \cF=\cR$ is reversible, $(\cR_{\cU}'\setminus \cR_{\cU})^{-1}= \cF\cap (\cR_{\cU}\setminus \cR_{\cU}')=\cF$ and $\cF \cap \cR_{\cU}'=\cR_0=\emptyset$. In particular \eqref{rvred0} - \eqref{rvred3} with $\cF_0=\cF$ of Theorem \ref{thm.rvs} are all fulfilled, but $\cR_{\cU,\cF}^*=\{S_1\ce{->} S_2+S_3\}$ is not weakly reversible.
\end{example}

The last theorem of this section concerns reachability of the original and reduced RNs.

\begin{theorem}\label{reard}
Let $\cR$ be an RN and assume $\cU\subseteq\cS$ is eliminable with respect to 
$\cF\subseteq \cR_\cU$, as in Definition \ref{def.red2}. Let $x,x'\in\N_0^n$.
\begin{enumerate}[(i)]
\setlength\itemsep{1em}
\item  If $x$\label{reard1} leads to $x'$ via $\cR_{\cU,\cF}^*$, then $x$ leads to $x'$.
\item \label{reard2} Reversely, suppose that  $\cU$ consists of intermediate species and $\cF=\cR_{\cU}$. Assume $(\supp(x)\cup \supp(x'))\cap \cU=\emptyset$. Then if $x$ leads to $x'$ via $\cR$, then $x$ leads also to $x'$ via $\cR_{\cU,\cF}^*$.
\end{enumerate}
\end{theorem}

\begin{proof}
\eqref{reard1} It follows directly from the definition of the reduced RN. 

\eqref{reard2} Suppose $x$ leads to $x'$ in $\cR$ and $(\supp(x)\cup \supp(x'))\cap \cU=\emptyset$. Then by Lemma \ref{leads} there are reactions $r_1\dots, r_m\in\cR$ (possibly with repetitions)  such that $\oplus_{i=1}^m r_i\leq (x,x')$ and $\oplus_{i=1}^m r_k\sim (x,x')$. Without loss of generality, assume $\oplus_{i=1}^m r_i= (x,x')$. 
\black{If this is not the case, then we proceed  with $(z,z')=\oplus_{i=1}^mr_i$, rather than $(x,x')$, and show that $(z,z')\in \cl(\cR_{\cU,\cF}^*)$. This subsequently implies that $x$ leads to $x'$ via $\cR_{\cU,\cF}^*$ as $\oplus_{i=1}^m r_k\sim (x,x')$. }

If $r_1,\dots, r_m\in \cR_0$, then $\oplus_{i=1}^m r_i\in \cl(\cR^*_{\cU,\cF})$, and we are done. Otherwise, since $\supp(x)\cap \cU=\emptyset$, by Lemma \ref{lem:non-int-red},  the reaction in $\{r_1,\dots,r_m\}\cap (\cR_{\cU}\cup \cR_{\cU}')$ with the smallest index belongs to  $\cR_{\cU}'\setminus \cR_{\cU}$. Without loss of generality, assume this reaction is $r_1=x_1\ce{->} u_1$, where $u_1\in\cU$ ($\cU$ consists of intermediate species).  \black{Otherwise, if $r_k$ is the first one, then $r_1,\dots, r_{k-1}\in \cR_0\subseteq \cR_{\cU,\cF}^*$, and we might define $r_1'=r_k, r_2'=r_{k+1}, \dots, r_{m-k+1}'=r_m$. Proceeding with the same argument as below, one can show that  $\oplus_{i=1}^{m-k+1} r_i'\in \cl(\cR_{\cU,\cF}^*)$, and thus $r_1\oplus \cdots \oplus r_{k-1}\oplus r_1'\oplus \cdots \oplus r_{m-k+1}' \in \cl(\cR_{\cU,\cF}^*)$ as well. Hence, we take $k=1$.}

Since $\supp(x')\cap \cU=\emptyset$, then there exists  $k\in \{2,\dots, m\}$, such that $u_1$ is the reactant of $r_k$, but not that of $r_2,\dots, r_{k-1}$. Let $r_{2:k-1}=(x_{2:k-1},x_{2:k-1}')=\oplus_{i=2}^{k-1} r_i$. We claim that 
\begin{align}\label{preard1}
r_1\oplus r_k\oplus r_{2:k-1}\leq \oplus_{i=1}^k r_i\ \mathrm{and}\ r_1\oplus r_k\oplus r_{2:k-1}\sim \oplus_{i=1}^k r_i.
\end{align}
The equivalence in \eqref{preard1} is a consequence of Theorem \ref{cmgp}. It suffices to show the inequality. Let $r_k=u_1\ce{->} x_2$, then $r_1\oplus r_k=(x_1,x_2)$
and thus
\[
r_1\oplus r_k\oplus r_{2:k-1}=\big(x_1+0\vee (x_{2:k-1}-x_2), x_{2:k-1}'+0\vee (x_2-x_{2:k-1})\big).
\]
On the other hand, by the choice of  $r_1$ and $r_k$, we have
\begin{align*}
\oplus_{i=1}^k r_i=r_1\oplus r_{2:k+1}\oplus r_k=&(x_1+x_{2:k-1}, u_1+x_{2:k-1}')\oplus (u_1, x_2)\\
=&(x_1+x_{2:k-1}, x_{2:k-1}'+x_2).
\end{align*}
This proves conclusion \eqref{preard1}. Note that $r_k=u_1\ce{->}x_2$ implies that either $x_2=u_2\in \cU$ or $\supp(x_2)\cap \cU=\emptyset$. Thus,  the procedure can be repeated to obtain $r_{\sigma(1)},\dots, r_{\sigma(m)}$, where $\sigma$ is a permutation of $\{1,\dots, m\}$, such that
\begin{align}\label{preard2}
\oplus_{i=1}^m r_{\sigma(i)}&\leq \oplus_{i=1}^m r_{i}=(x,x'),\quad \oplus_{i=1}^m r_{\sigma(i)}\sim(x,x'),
\end{align}
\black{which is implied by the fact that $\cU$ consists of intermediate species.} Moreover, there exist $0=k_0< k_1< \dots < k_j< k_{j+1}= m$, such that for each $i=0,\dots, j$, either $r_{\sigma(k_i+1)},\dots, r_{\sigma(k_{i+1})}\in\cR_0$, or $r_{\sigma(k_i+1)}\in \cR_{\cU}'\setminus \cR_{\cU}$, $r_{\sigma(k_i+1)},\dots, r_{\sigma(k_{i+1}-1)}\in \cR_{\cU}\cap \cR_{\cU}'$ and $r_{\sigma(k_{i+1})}\in \cR_{\cU}\setminus\cR_\cU'$ with $r_{\sigma(k_i+1)}\oplus \ldots \oplus r_{\sigma(k_{i+1})}\in \overline{\cR}_0$. Therefore, $r_{\sigma(k_i+1)}\oplus\dots\oplus r_{\sigma(k_{i+1})}\in \cl(\cR_{\cU,\cF}^*)$ for all $i=1,\dots, j$. This yields  $\oplus_{i=1}^m r_{\sigma(i)}\in \cl(\cR_{\cU,\cF}^*)$ as well. Combining \eqref{preard2} and Lemma \ref{leads}, it follows that $x$ leads to $x'$ via $\cR_{\cU,\cF}^*$. The proof is complete.
\end{proof}

Theorem \ref{reard}\eqref{reard2} does not hold in general, not even for non-interacting species. Consider the following counterexample,
\[
\cR=\{S_1\ce{->} S_2+U\ce{->} S_3, S_2\ce{->} S_4, S_4+U\ce{->} S_5\}
\]
with $\cU=\{U\}$ and $\cF=\cR_{\cU}$. Then, 
\[
\cR_{\cU,\cF}^*=\{ S_2\ce{->} S_4,S_1\ce{->} S_3,S_1+S_4\ce{->} S_2+S_5\}.
\]
Note that $(S_1,S_5)=(S_1\ce{->} S_2+U)\oplus  (S_2\ce{->} S_4)\oplus (S_4+U\ce{->} S_5)$. Thus $S_1$ leads to $S_5$ via $\cR$, but not via $\cR_{\cU,\cF}^*$.

\section{Discussion and conclusion}

 We introduced and analysed the properties of a  sum operation on chemical reactions. Thereby, we  connect and characterise structural properties of RNs, such as reachability, (weakly) reversibility, and being essential via the closure of the sum operation. This extends previous characterisations \cite{Cappelletti,jmb-14-pauleve-craciun-koeppl,Mreach} and connects such properties to the geometry of the closure $\cl(\cR)$ in the product space $\N_0^n\times \N_0^n$. In another direction, we defined reductions of RNs by elimination of species from an RN by adding reactions. Those reductions originate from connections to the slow-fast limits of stochastic RNs \cite{CW16}. Furthermore, we studied the conservation of (weakly) reversibility, when reachability of the original and the reduced network coincide in some sense.

 As the discrete dynamics of Petri Nets and  vector addition systems correspond directly to dynamics of RNs \cite{Cook2009}, the developed theory pertains to those areas as well. Correspondingly, problems and questions from theoretical computer science relate to the notions we have introduced. As an example, an undecidable problem relating to Section \ref{sec.appl} asks whether two RNs given by their reaction sets $\cR_1,\cR_2$ with initial values $x_1,x_2$, respectively, have the same reachability sets, i.e. whether $\cR_1(x_1)=\cR_2(x_2)$ \cite{HACK}. 
 Another example is the decidable reachability problem that asks whether given an RN and two states $x_1,x_2$, we can reach $x_2$ from $x_1$  \cite{Cook2009,Mreach}.

Furthermore, the closure $\cl(\cR)$ of an RN  has only sometimes the structure of a semi-linear set. This is not surprising as the set of reachable states of an RN directly relates to the closure  $\cl(\cR)$ of $\cR$, see Section \ref{sec.appl}. Reachability sets can be highly complex and are not necessarily semi-linear \cite{HOPCROFT1979135,normal_petri}. Nonetheless, it might be interesting to characterise and study the structure of RNs $\cR$ for which $\cl(\cR)$ is semi-linear.

Overall we hope that the sum calculus on reactions we have introduced will find further applications, possibly even in areas which  a priori are not directly linked to our areas of research.

 \bibliographystyle{plain}

 \bibliography{references} 
 
\end{document}